\documentclass[11pt]{article}

\usepackage{cmap}
\usepackage[T1]{fontenc}
\usepackage[utf8]{inputenc}
\usepackage{lmodern}
\usepackage[margin=1.15in]{geometry}
\usepackage{microtype}
\usepackage{amsmath,amssymb,amsthm,mathtools,mathrsfs}
\usepackage{enumitem}
\usepackage{needspace}
\usepackage[dvipsnames]{xcolor}
\usepackage[
  backend=biber,
  style=alphabetic,
  sorting=nyt,
  maxbibnames=99,
  maxalphanames=4,
  natbib=true
]{biblatex}
\appto{\bibsetup}{\emergencystretch=1em}
\usepackage{aliascnt}
\usepackage{hyperref}
\usepackage[nameinlink,noabbrev,capitalise]{cleveref}

\newtheorem{theorem}{Theorem}[section]
\newaliascnt{proposition}{theorem}
\newtheorem{proposition}[proposition]{Proposition}
\aliascntresetthe{proposition}
\newaliascnt{lemma}{theorem}
\newtheorem{lemma}[lemma]{Lemma}
\aliascntresetthe{lemma}
\newaliascnt{corollary}{theorem}
\newtheorem{corollary}[corollary]{Corollary}
\aliascntresetthe{corollary}
\theoremstyle{definition}
\newaliascnt{example}{theorem}
\newtheorem{example}[example]{Example}
\aliascntresetthe{example}
\theoremstyle{remark}
\newaliascnt{remark}{theorem}

\aliascntresetthe{remark}

\setlist[itemize]{leftmargin=2em,itemsep=0.3em,topsep=0.4em}
\setlist[enumerate]{leftmargin=2em,itemsep=0.3em,topsep=0.4em}
\allowdisplaybreaks
\numberwithin{equation}{section}

\crefname{subsection}{Subsection}{Subsections}
\Crefname{subsection}{Subsection}{Subsections}

\newcommand{\R}{\mathbb R}
\newcommand{\ee}{\mathrm e}
\newcommand{\dd}{\mathrm d}

\title{Projective Rigidity of $U(k)$-Radial Ricci-Flat K\"ahler Metrics}
\author{Shaosai Huang\thanks{Kspectra Research Inc., Toronto, ON, M2N 0G3,
Canada. Email:
\href{mailto:arthur.foxie.huang@kspectra.ai}{arthur.foxie.huang@kspectra.ai}.}}
\date{November 3, 2026}
\hypersetup{
  colorlinks=true,
  linkcolor=MidnightBlue,
  citecolor=OliveGreen,
  urlcolor=BrickRed,
  pdftitle={Projective Rigidity of U(k)-Radial Ricci-Flat Kähler Metrics},
  pdfauthor={Shaosai Huang},
  pdfsubject={Projective rigidity of Ricci-flat Kähler metrics on direct sums of canonical-bundle roots},
  pdfkeywords={Ricci-flat Kähler metrics, Kähler immersions, projectively induced metrics, Calabi ansatz, Taub–NUT metric, resolved conifold},
  pdflang={en-US}
}

\begin{document}
\maketitle

\begin{abstract}
We study K\"ahler immersions of Ricci-flat metrics into finite- or
infinite-dimensional complex projective space.  Let $L^{\otimes k}\simeq K_X$
over a real-analytic K\"ahler manifold $X$.  Let $\phi$ be a base potential
and $s$ the canonical squared fibre norm on $L^{\oplus k}$.
Consider Ricci-flat potentials $\phi+h(s)$ near the zero section,
with $h$ independent of the base point.
We prove that local projective inducibility forces the base to be
Ricci-flat and the metric to split locally as $X\times\mathbb C^k$.
In that case projective inducibility is equivalent to that of the base.
The proof combines Calabi's criterion, positivity of the homogeneous
fibre-degree blocks, and the singularities of the radial profile.
A related obstruction along holomorphic curves excludes every positive
multiple of the Taub--NUT metric and Calabi's
hyperk\"ahler metrics on $T^{*1,0}\mathbb{CP}^n$.  Applications also include the
resolved conifold, canonical-bundle metrics in non-Einstein classes,
and Cifarelli's rank-two Type-2 metric.
\end{abstract}



\section{Introduction}
\label{sec:intro}

A K\"ahler metric is \emph{projectively induced} if it is the pullback of
the Fubini--Study metric by a holomorphic immersion.  Loi, Salis and
Zuddas conjectured that every Ricci-flat projectively induced K\"ahler
metric is flat \cite[Conjecture~1]{LoiSalisZuddas2018}.  The question is
local and includes immersions into $\mathbb{CP}^{\infty}$, the projective
space of a separable complex Hilbert space.  We normalize the
Fubini--Study form by
\begin{equation*}
 \omega_{\mathrm{FS}}
 =i\partial\bar\partial\log\left(1+\sum_j|Z_j|^2\right)
\end{equation*}
in affine coordinates.  A metric is \emph{locally projectively induced}
near a point if some neighbourhood admits such an immersion.

The resolved conifold provides a concrete instance of the problem.
The total space of $\mathcal O_{\mathbb{CP}^1}(-1)^{\oplus2}$ carries
the complete Ricci-flat metrics of Candelas and de la Ossa
\cite{CandelasDeLaOssa1990}.  Their potentials depend on the fibre
coordinates through a single squared norm.  We show that no positive
multiple of these metrics admits a K\"ahler immersion into
$\mathbb{CP}^{\infty}$, even on a nonempty open subset.  This follows
from a rigidity theorem for the same construction in arbitrary fibre rank.

\subsection{Main Result and Applications}

Let $(X,\omega)$ be a real-analytic K\"ahler manifold and
$L\to X$ a holomorphic line bundle with $L^{\otimes k}\simeq K_X$, where $k\geq1$ and
$K_X$ is the canonical bundle.  Give $L$ the Hermitian metric induced by
$\omega$ through this isomorphism, and set $E=L^{\oplus k}$.
If $\ell$ is the squared norm of a local frame of $L$,
$\boldsymbol\xi\in\mathbb C^k$ is the fibre coordinate on $E$, and
$\phi$ is a potential for $\omega$, the fibrewise $U(k)$-radial ansatz is
\begin{equation}\label{eq:uk-potential}
 \widetilde\phi_k=\phi+h(s),\qquad
 s=\ell|\boldsymbol\xi|^2,\qquad h(0)=0,\qquad h'(0)=1.
\end{equation}
Here $h$ is real analytic near $0$ and independent of the base point;
the normalization $h'(0)=1$ fixes fibre dilation.  This is a classical radial
ansatz \cite{Calabi1979,HwangSinger2002}.  Write $\rho$ for the Ricci form
of $\omega$.  The base has \emph{constant principal Ricci curvatures} if
the eigenvalues of the Ricci endomorphism $\omega^{-1}\rho$ are independent
of the base point.  This includes homogeneous metrics and products of
K\"ahler--Einstein metrics with different Einstein constants.

Our main result states that a projectively induced extension of this
form must split locally with a flat fibre.

\begin{theorem}[$U(k)$-Radial Projective Rigidity]
\label{thm:uk-rigidity}
Let $(X,\omega)$ be a real-analytic K\"ahler manifold, let $k\geq1$ be an
integer, let
$L^{\otimes k}\simeq K_X$, and
suppose that
$\widetilde\omega_k=i\partial\bar\partial\widetilde\phi_k$ is a normalized
$U(k)$-radial Ricci-flat germ of the form
\eqref{eq:uk-potential}, defined on a connected neighbourhood $W$ of the
zero section.  Then the base has constant principal Ricci curvatures.  If,
for some $c>0$, some nonempty open subset of $(W,c\widetilde\omega_k)$
admits a K\"ahler immersion into $\mathbb{CP}^N$, where
$N\in\mathbb N\cup\{\infty\}$, then
$\rho=0$, and in a local unitary parallel frame of $L$
\begin{equation}\label{eq:uk-product}
 \widetilde\omega_k
 =\omega+i\partial\bar\partial|\boldsymbol\xi|^2.
\end{equation}
Conversely, if $\rho=0$, then, for every $c>0$,
$c\widetilde\omega_k$ is locally projectively induced exactly when
$c\omega$ is.  Whenever these equivalent conditions hold, Calabi's
coefficient matrix has infinite rank; hence every realizing projective target is
infinite-dimensional.
\end{theorem}

The theorem is proved in \cref{sec:uk-rigidity}; Calabi's coefficient
matrix is recalled in \cref{sec:obstruction}.  The conclusion reduces
the conjecture to a base of smaller complex dimension; it does not assert
flatness of an arbitrary Ricci-flat base.

The complete radial applications include the resolved conifold,
the unequal-radius Pando Zayas--Tseytlin metrics, van Coevering's
canonical-bundle metrics in arbitrary invariant K\"ahler classes,
and the equal-summand metrics of Achmed-Zade and Bykov
(\cref{cor:classical-complete}).  The scalar obstruction used in the proof,
applied along suitable holomorphic curves, excludes every positive multiple
of LeBrun's Taub--NUT metric, Calabi's hyperk\"ahler metrics on
$T^{*1,0}\mathbb{CP}^n$, and Cifarelli's rank-two Type-2 metric.

\subsection{Mechanism and Prior Work}

Calabi's classical criterion expresses projective inducibility as
positivity of an infinite matrix of Taylor coefficients
\cite{Calabi1953}; we recall it in \cref{sec:obstruction}.
For the bundle ansatz, this matrix decomposes by homogeneous fibre degree.
For a positive scale $c$, write
\begin{equation*}
 \exp(ch(s))=\sum_{n\geq0}f_{c,n}s^n.
\end{equation*}
Positivity gives $f_{c,n}\geq0$, and for every $n$ with $f_{c,n}>0$
it forces $c\omega+(n/k)\rho$ to be semipositive.  The positive
coefficients occur in arbitrarily large degrees, excluding negative Ricci
eigenvalues.  A positive eigenvalue produces a finite negative-real
branch point of the inverse profile, while the exponential potential
continues across the positive real axis.  The Vivanti--Pringsheim theorem
rules out this combination for a series with nonnegative coefficients.
\Cref{lem:one-variable} formulates this scalar obstruction, and
\cref{prop:radial-curve} applies it to radial potentials on holomorphic
curves.  The bundle proof also uses the horizontal curvature inequalities;
these are not detected by restriction to a single fibre.

Both the coefficient criterion and related fibre-degree decompositions
are established tools.  The latter occur for Hartogs domains
\cite{LoiZedda2011MathAnn,HaoWang2014,LoiZeddaBook2018} and for line-bundle
metrics \cite{LoiZeddaZuddas2021,CristoforiZedda2024}.
Loi, Zedda and Zuddas excluded Calabi's Ricci-flat canonical-bundle
metrics over compact K\"ahler--Einstein bases \cite{LoiZeddaZuddas2021}.
Cristofori and Zedda showed that the lowest-order fibre coefficient in
Calabi's criterion already excludes these metrics for a range of small
multiples \cite[Lemma~1 and Remark~6]{CristoforiZedda2024}.
Marini, Tardini and Zedda showed that a K\"ahler cone over a regular
complete Sasakian manifold is Ricci-flat and projectively induced only if
it is flat \cite[Theorem~1]{MariniTardiniZedda2024}.
Arezzo, Loi, Placini and Zedda proved rigidity for K\"ahler--Einstein
metrics with a radial potential in local holomorphic coordinates
\cite[Theorem~1.1(3)]{ArezzoLoiPlaciniZedda2026}; their Lemma~4.1
also uses the branch-point mechanism.  Here radial symmetry is confined
to the fibres over a possibly non-Einstein base.  The curvature identity
for the canonical root couples the fibre blocks to $\rho$, yielding
the splitting theorem in every fibre rank.

For Taub--NUT, Loi, Zedda and Zuddas proved nonexistence in a restricted
range of positive multiples and conjectured the result at every scale
\cite{LoiZeddaZuddas2012,LoiZeddaBook2018}.  The first version of this
paper settled the remaining case \cite{Huang2026TaubNut}; we include the
short curve proof in \cref{thm:taub-nut}.

\subsection{Scope and Organization}

The main theorem concerns the normalized potential $\phi+h(s)$;
base-dependent radial profiles and general torus-invariant extensions
are further problems.  The additional nonexistence conclusions concern
the infinite-dimensional target.  For finite-dimensional targets,
Hulin's extension and positivity theorems
\cite{Hulin1996,Hulin2000}, together with
\cite[Corollary~1.7 and the subsequent remark]{ArezzoLiLoi2025},
already exclude Ricci-flat K\"ahler immersions.

\Cref{sec:obstruction} establishes the common one-variable obstruction.
\Cref{sec:uk-rigidity} constructs the radial profile, derives the
fibre-degree inequalities, and proves \cref{thm:uk-rigidity}.
\Cref{sec:canonical} identifies the complete radial applications, and
\cref{sec:curves} treats the remaining metrics along holomorphic curves.
\Cref{sec:directions} records the main open extension.

\section{Calabi's Criterion and a One-Variable Obstruction}
\label{sec:obstruction}

Both geometric arguments use the same consequence of coefficient
positivity.  We recall Calabi's criterion and state the common obstruction.

\subsection{Calabi's Criterion}

For a real-analytic potential $\Phi$ near $p$, choose holomorphic
coordinates with $z(p)=0$.  Its polarization $\Phi(z,\bar w)$ is the
unique local extension holomorphic in $z$ and antiholomorphic in $w$.
Calabi's diastasis centred at $p$ is the potential whose polarization is
\begin{equation*}
 D_p(z,\bar w)
 =\Phi(z,\bar w)+\Phi(0,0)-\Phi(z,0)-\Phi(0,\bar w).
\end{equation*}
The pluriharmonic ambiguity cancels, and $D_p$ has no nonconstant pure
holomorphic or antiholomorphic terms.  For a real-valued real-analytic
function $K$, write $C(K)$ for the Hermitian coefficient matrix of its
polarization, indexed by all holomorphic monomials, including the
constant.  Calabi's criterion says that a metric is locally projectively
induced near $p$ precisely when $C(\ee^{D_p}-1)\succeq0$; the rank is the
least projective target dimension.  For $c$ times the metric, replace
$D_p$ by $cD_p$ \cite{Calabi1953}.

We also use the following local propagation property of Calabi's
criterion.

\begin{lemma}[Global Character of Projective Inducibility]
\label{lem:global-character}
Let $(M,g)$ be a connected real-analytic K\"ahler manifold.  If some
nonempty open subset of $M$ admits a K\"ahler immersion into
$\mathbb{CP}^N$, $N\leq\infty$, then every point of $M$ has a neighbourhood
admitting a K\"ahler immersion into $\mathbb{CP}^N$.
\end{lemma}

\begin{proof}
By Calabi's criterion, a point $p$ has such a neighbourhood exactly when
$C(\ee^{D_p}-1)$ is positive semidefinite of rank at most $N$.  Calabi
showed that, for each $N'\leq\infty$, the set of points at which this
matrix is positive semidefinite of rank $N'$ is open and closed in $M$
\cite{Calabi1953}; see \cite[Section~2.2]{LoiZeddaBook2018} for a detailed
account covering $N=\infty$.
\end{proof}

\subsection{The One-Variable Obstruction}

For a radial potential on a disc, Calabi's matrix is diagonal.
The following elementary consequence of the Vivanti--Pringsheim theorem
forces a negative coefficient at every positive scale.

\begin{lemma}[One-Variable Pringsheim Obstruction]
\label{lem:one-variable}
Let $R>0$, and let $\Psi$ be holomorphic on a connected open neighbourhood
$U\subset\mathbb C$ of the half-line $(-R,\infty)$, real on $(-R,\infty)$,
with $\Psi(0)=0$.  Suppose that $\Psi$ is bounded on $(-R,0]$ and that some
derivative $\Psi^{(m)}$, $m\geq1$, is unbounded on $(-R,0]$.  Then, for
every $c>0$, the Taylor expansion
$\ee^{c\Psi(x)}=\sum_{j\geq0}a_jx^j$ at $x=0$ has a negative coefficient.
\end{lemma}

\begin{proof}
Let $r_0$ be the radius of convergence of the series and $F$ its sum
on $\Delta=\{|x|<r_0\}$.  By the identity theorem,
$F=\ee^{c\Psi}$ on the connected component $\Omega$ of
$\Delta\cap U$ containing $0$.  Suppose first that $r_0>R$.
Then $(-R,0]\subset\Omega$ and $F$ is holomorphic at $-R$.  Since $\Psi$
is bounded on $(-R,0]$, $F(-R)=\lim_{x\to-R^+}\ee^{c\Psi(x)}\neq0$, so on
a small disc $B$ about $-R$ the function $c^{-1}\log F$, with the branch
equal to $\Psi$ on $B\cap(-R,0]$, is holomorphic.  All its derivatives are
bounded near $-R$, contradicting the hypothesis on $\Psi^{(m)}$.  Hence
$r_0\leq R<\infty$.

Suppose now that every $a_j\geq0$.  By the Vivanti--Pringsheim theorem
\cite{Vivanti1893,Pringsheim1894}, \cite[Theorem~IV.6]{FlajoletSedgewick2009},
the point $x=r_0$ is a
singular point of $F$.  But if $B\subset U$ is a small disc about
$r_0$, then $B\cap\Delta$ is connected and meets
$[0,r_0)\subset\Omega$, hence lies in $\Omega$; there
$F=\ee^{c\Psi}$, so $\ee^{c\Psi}$ continues $F$ holomorphically across
$r_0$.  This contradiction proves the lemma.
\end{proof}

Calabi's criterion turns the lemma into an obstruction along holomorphic
curves.

\Needspace{7\baselineskip}
\begin{proposition}[Radial Holomorphic Curves]
\label{prop:radial-curve}
Let $(M,g)$ be a connected real-analytic K\"ahler manifold.  Suppose that
through some point $p\in M$ there is an embedded holomorphic disc, with
coordinate $t$ centred at $p$, on which $g$ has a K\"ahler potential
$\Psi(\lambda|t|^2)$, where $\lambda>0$ and $\Psi$ is the germ at $0$ of a
function satisfying the hypotheses of \cref{lem:one-variable}.  Then,
for every $c>0$, no nonempty open subset of $(M,cg)$ admits a K\"ahler
immersion into $\mathbb{CP}^{\infty}$.
\end{proposition}

\begin{proof}
Otherwise \cref{lem:global-character} gives a K\"ahler immersion of a
neighbourhood of $p$, and its restriction to the disc is a K\"ahler
immersion of the induced metric near $t=0$.  Since $\Psi(0)=0$, the
polarization $\Psi(\lambda t\bar w)$ has no pure terms at $0$, so
$c\Psi(\lambda|t|^2)$ is the diastasis of $c\,g$ on the disc at $0$.  Its
Calabi matrix is diagonal in the monomials $t^j$, with entries
$a_j\lambda^j$, where $\ee^{c\Psi}=\sum_ja_jx^j$.  Calabi's criterion
therefore forces every $a_j\geq0$, contradicting
\cref{lem:one-variable}.
\end{proof}

\Cref{prop:radial-curve} is a variant of Theorem~B in the first version
of this paper \cite{Huang2026TaubNut}, in which \cref{lem:one-variable}
derives the required finite Taylor radius from a singularity on the
negative axis; no global size of the curve is needed.

\section{Rigidity of Radial Bundle Metrics}
\label{sec:uk-rigidity}

We prove \cref{thm:uk-rigidity} by combining the Ricci-flat profile
equation with positivity of the homogeneous fibre-degree blocks.
The blocks first exclude negative base Ricci eigenvalues;
\cref{lem:one-variable} then excludes positive ones.

\subsection{The Radial Ricci-Flat Profile}
\label{sec:uk-extension}

We first determine the normalized Ricci-flat germ.  Let $(X^r,\omega)$
be a connected real-analytic K\"ahler manifold.  On a coordinate ball choose a potential $\phi$ and
write
\begin{equation}\label{eq:uk-base}
 \omega=i\partial\bar\partial\phi,
 \qquad
 \rho=-i\partial\bar\partial\log\det(g_{a\bar b}),
 \qquad A=\omega^{-1}\rho.
\end{equation}
Here $(g_{a\bar b})$ is the Hermitian matrix of $\omega$ in the chosen
coordinates.  Determinants of base endomorphisms are taken on $T^{1,0}X$.  Assume that a holomorphic line bundle $L$ is supplied with
an isomorphism $L^{\otimes k}\simeq K_X$, and give $L$ the Hermitian
metric whose $k$th tensor power is the metric on $K_X$ induced by
$\omega$.  If $\ell$ denotes the
squared norm of a local holomorphic frame of $L$, then, after multiplying
the frame by a nowhere-zero holomorphic function,
\begin{equation}\label{eq:root-metric}
 \ell^k=\det(g_{a\bar b})^{-1},
 \qquad i\partial\bar\partial\log\ell=\frac1k\rho.
\end{equation}
For $\boldsymbol\xi=(\xi_1,\ldots,\xi_k)$ in $E=L^{\oplus k}$ put
\begin{equation}\label{eq:radial-s}
 s=|\boldsymbol\xi|_L^2
 =\ell\sum_{j=1}^k|\xi_j|^2.
\end{equation}
For the normalized potential \eqref{eq:uk-potential}, put
\begin{equation}\label{eq:radial-moment}
 q(s)=sh'(s).
\end{equation}
The normalization $h'(0)=1$ is achieved by a dilation of the fibres, and
$q$ is the moment map of the circle of scalars in $U(k)$.

The following computation is the canonical-root specialization of the
momentum construction \cite[Sections~2 and~3.2]{HwangSinger2002}.
In rank one, the compatibility condition already appears in
\cite[proof of Proposition~2.2]{Bielawski2002}.  We include the local
calculation to fix the factors of $k$.

\begin{lemma}[Radial Determinant and the Compatibility Condition]
\label{lem:uk-determinant}
For the potential \eqref{eq:uk-potential}, the determinant of the
Hermitian matrix of $i\partial\bar\partial\widetilde\phi_k$ is the product
of the horizontal factor
\begin{equation*}
 \det(g_{a\bar b})\det(I+(q/k)A)
\end{equation*}
and the vertical factor
\begin{equation*}
 \ell^k\left(\frac qs\right)^{k-1}q'.
\end{equation*}
Write
\begin{equation*}
 P_x(t)=\det(I+tA_x).
\end{equation*}
If $i\partial\bar\partial\widetilde\phi_k$ is Ricci-flat, then $P_x$ is
independent of $x$; equivalently, the base has constant principal Ricci
curvatures.  Once this condition holds, the ansatz is Ricci-flat exactly
when
\begin{equation}\label{eq:uk-ode}
 P(q/k)\left(\frac qs\right)^{k-1}q'=1,
 \qquad P(t)=P_x(t).
\end{equation}
\end{lemma}

\begin{proof}
At a chosen base point take a holomorphic frame whose Chern connection
one-form vanishes.  The horizontal block is
 $g+(q/k)\rho$.  The vertical block is
\begin{equation*}
 \ell\left(h'I_k+h''\ell\,
 \overline{\boldsymbol\xi}\boldsymbol\xi^{\mathsf T}\right);
\end{equation*}
its eigenvalues are $\ell q/s$ with multiplicity $k-1$ and
$\ell q'$.  The Schur complement gives the stated product in arbitrary
coordinates.  Equation \eqref{eq:root-metric} cancels
$\det(g_{a\bar b})\ell^k$.  Let $H$ denote the remaining smooth positive
$U(k)$-radial factor.  Ricci-flatness makes $\log H$ pluriharmonic.
Restricting to each fibre, a smooth $U(k)$-radial pluriharmonic function on
a neighbourhood of the origin in $\mathbb C^k$ is constant.  Thus $H$ is
fibre-constant; its value along the zero section is one because
$q(s)/s\to1$ and $q'(s)\to1$.  Hence $H\equiv1$, and therefore
\begin{equation}\label{eq:uk-pointwise-ode}
 P_x(q(s)/k)
 =\left(\frac{s}{q(s)}\right)^{k-1}\frac1{q'(s)}.
\end{equation}
The right side is independent of $x$.  Since $q(s)=s+O(s^2)$, the map
$s\mapsto q(s)/k$ is locally invertible.  Thus all the polynomials $P_x(t)$
agree near $t=0$, and hence identically.  This proves the compatibility
condition.  Conversely, if $P_x=P$ is base-independent and
\eqref{eq:uk-ode} holds, the determinant factor is one, so the ansatz is
Ricci-flat.
\end{proof}

Integrating \eqref{eq:uk-ode} gives the germ explicitly.

\begin{proposition}[Existence and Uniqueness of the Normalized Radial Extension]
\label{prop:uk-existence}
The normalized $U(k)$-radial ansatz \eqref{eq:uk-potential} admits a
Ricci-flat germ if and only if the base has constant principal Ricci
curvatures.  When it exists, it is unique within this ansatz and is
determined by
\begin{equation}\label{eq:uk-integrated}
 s^k=Q_k(q),
 \qquad
 Q_k(q)=k\int_0^q u^{k-1}P(u/k)\,\dd u.
\end{equation}
For $k=1$ this coincides with the constant-principal-Ricci-curvature
specialization of Bielawski's extension on $K_X$.
\end{proposition}

\begin{proof}
Necessity is \cref{lem:uk-determinant}.  For sufficiency, let $P$ be
the base-independent characteristic polynomial.  Multiplying
\eqref{eq:uk-ode} by $s^{k-1}$ and integrating gives
\eqref{eq:uk-integrated}.  Near zero,
\begin{equation*}
 Q_k(q)=q^kB(q),\qquad B(0)=1.
\end{equation*}
The analytic map $q\mapsto qB(q)^{1/k}$ has derivative one at zero and is
therefore locally invertible.  Its inverse is the unique analytic
$q(s)=s+O(s^2)$.  Then $h'(s)=q(s)/s$ is analytic and
\eqref{eq:uk-potential} determines $h$.  Positivity holds after shrinking
the fibre neighbourhood.  When $k=1$, \eqref{eq:uk-integrated} becomes
$s=\int_0^qP(u)\,\dd u$; Bielawski's uniqueness theorem
\cite[Theorem~1]{Bielawski2002} identifies the resulting circle-invariant
germ on $K_X$ with his extension, after a fibre dilation.
\end{proof}

\subsection{Fibre-Degree Positivity}

The fibre symmetry makes the coefficient matrix block diagonal.  The
next proposition gives the exact blocks and the curvature consequence
needed below.

\begin{proposition}[Fibre-Degree Criterion]
\label{prop:fibre-degree}
Let $(X,\omega)$ be a real-analytic K\"ahler manifold, let $L\to X$ be a
holomorphic line bundle with a real-analytic Hermitian metric, let $\ell$
be its squared norm in a local frame, and put
$\Theta=i\partial\bar\partial\log\ell$.  Let $h$ be real analytic near
$0$ with $h(0)=0$, and suppose that
\begin{equation*}
 \widetilde\omega
 =i\partial\bar\partial\bigl(\phi+h(\ell|\boldsymbol\xi|^2)\bigr),
 \qquad \omega=i\partial\bar\partial\phi,
\end{equation*}
is a K\"ahler metric on a neighbourhood of the zero section of
$L^{\oplus k}$.  Fix $c>0$ and $p\in X$, and write
$\ee^{ch}=\sum_{n\geq0}f_{c,n}s^n$.  Take $\phi$ to be the diastasis of
$\omega$ at $p$, so that its polarization satisfies
$\phi(z,\bar p)\equiv0$, choose the frame of $L$ so that the polarization
of $\ell$ satisfies $\ell(z,\bar p)\equiv1$, and set
$K_n=\ee^{c\phi}\ell^n$.  Then $c\widetilde\omega$ is projectively induced
near $(p,0)$ if and only if
\begin{enumerate}[label=\textup{(\roman*)}]
\item $c\omega$ is projectively induced near $p$;
\item $f_{c,n}\geq0$ for every $n\geq1$; and
\item $C(K_n)\succeq0$ for every $n\geq1$ with $f_{c,n}>0$.
\end{enumerate}
If $C(K_n)\succeq0$, then $c\omega+n\Theta\succeq0$ near $p$; when
$c\omega+n\Theta$ is positive definite at $p$, condition~(iii) for $n$
says exactly that $c\omega+n\Theta$ is projectively induced near $p$.
When \textup{(i)}--\textup{(iii)} hold, the rank of Calabi's coefficient
matrix for $c\widetilde\omega$ at $(p,0)$ is
\begin{equation}\label{eq:fibre-rank}
 \operatorname{rank}C(\ee^{c\phi}-1)
 +\sum_{n\geq1,\ f_{c,n}>0}\binom{n+k-1}{k-1}\operatorname{rank}C(K_n).
\end{equation}
\end{proposition}

\begin{proof}
Let $o=(p,0)$.  Since $h(0)=0$ and $\phi(z,\bar p)=0$, the polarized
potential
\begin{equation*}
 \phi(z,\bar w)+h(\ell(z,\bar w)\langle\xi,\eta\rangle)
\end{equation*}
has no pure terms at $o$, so it is the diastasis $D_o$ and
\begin{equation}\label{eq:polarized-normal-blocks}
 \ee^{cD_o((z,\xi),(\bar w,\bar\eta))}
 =\sum_{n\geq0}f_{c,n}K_n(z,\bar w)
   \langle\xi,\eta\rangle^n.
\end{equation}
Only monomials of equal fibre bidegree occur, so the Calabi matrix
$C(\ee^{cD_o}-1)$ is block diagonal by fibre degree.  The degree-zero block
is $C(\ee^{c\phi}-1)$, which is Calabi's matrix for $c\omega$ at $p$
because $c\phi$ is its diastasis; its positivity is~(i).  In the
normalized basis $u_\alpha=(n!/\alpha!)^{1/2}\xi^\alpha$ of
$\operatorname{Sym}^n(\mathbb C^k)$ one has
$\langle\xi,\eta\rangle^n=\sum_{|\alpha|=n}u_\alpha(\xi)
\overline{u_\alpha(\eta)}$, so the degree-$n$ block is
\begin{equation*}
 f_{c,n}\,C(K_n)\otimes I_{\operatorname{Sym}^n(\mathbb C^k)}
 \qquad(n\geq1).
\end{equation*}
The entry of $C(K_n)$ at the constant monomial is $K_n(p,\bar p)=1$.
Hence this block is positive semidefinite exactly when $f_{c,n}\geq0$ and,
in case $f_{c,n}>0$, $C(K_n)\succeq0$; this is (ii)--(iii).  A
block-diagonal Hermitian form is positive semidefinite exactly when every
block is, and its rank is the sum of the ranks of the blocks, which gives
the equivalence and \eqref{eq:fibre-rank}.

If $C(K_n)\succeq0$, factor $K_n(z,\bar w)=\langle G_n(z),G_n(w)\rangle$
with $G_n$ holomorphic into a Hilbert space.  For every tangent vector
$u$, Cauchy--Schwarz gives
\begin{equation*}
 \partial_u\bar\partial_u\log K_n(z,\bar z)
 =\frac{\|G_n\|^2\|\partial_uG_n\|^2
        -|\langle\partial_uG_n,G_n\rangle|^2}{\|G_n\|^4}\geq0,
\end{equation*}
while $i\partial\bar\partial\log K_n=c\omega+n\Theta$.  Finally,
$\log K_n=c\phi+n\log\ell$ has no pure terms at $p$, so it is the
diastasis of $c\omega+n\Theta$ at $p$ whenever this form is positive
definite.  Since $K_n(z,\bar p)\equiv1$, the matrix $C(K_n)$ is the direct
sum of $[1]$ and $C(K_n-1)$, and Calabi's criterion identifies~(iii) with
projective inducibility of $c\omega+n\Theta$ near $p$.
\end{proof}

For the induced Hermitian metric on $L$, the criterion gives the hypotheses of
\cref{lem:positive-inverse-profile}.

\begin{corollary}[Positivity Along the Zero Section]
\label{cor:uk-blocks}
If $c\,i\partial\bar\partial\widetilde\phi_k$ is projectively induced near
a point $(p,0)$ of the zero section, then
\begin{equation}\label{eq:uk-block-positivity}
 f_{c,n}\geq0,
 \qquad
 f_{c,n}>0\Longrightarrow
 c\omega+\frac nk\rho\succeq0\quad\text{near }p.
\end{equation}
\end{corollary}

\begin{proof}
Apply \cref{prop:fibre-degree} to the Hermitian metric
\eqref{eq:root-metric}, for which $\Theta=\rho/k$.
\end{proof}

\subsection{Completion of the Rigidity Proof}

The block inequalities constrain both the signs of the Ricci eigenvalues
and the analytic continuation of the fibre profile.  The following
one-variable statement packages these two consequences.

\begin{lemma}[One-Variable Rigidity of the Profile]
\label{lem:positive-inverse-profile}
Let
\begin{equation*}
 P(t)=\prod_{j=1}^r(1+\lambda_jt),\qquad \lambda_j\in\R,
\end{equation*}
and let the normalized germs $q,h$ satisfy
\begin{equation*}
 s^k=k\int_0^{q(s)}u^{k-1}P(u/k)\,\dd u,
 \qquad q=sh'(s),\qquad h(0)=0,\qquad h'(0)=1.
\end{equation*}
Fix $c>0$ and write $F_c=\ee^{ch}=\sum_{n\geq0}f_{c,n}s^n$.
Suppose that $f_{c,n}\geq0$ for every $n$ and that
\begin{equation}\label{eq:inverse-profile-positivity}
 f_{c,n}>0\quad\Longrightarrow\quad
 c+\frac nk\lambda_j\geq0
 \quad\text{for every }j.
\end{equation}
Then every $\lambda_j$ vanishes and
\begin{equation*}
 q(s)=s,\qquad h(s)=s,\qquad F_c(s)=\ee^{cs}.
\end{equation*}
\end{lemma}

\begin{proof}
First, $F_c$ is not a polynomial.  Indeed,
\begin{equation}\label{eq:uk-qF}
 q(s)=\frac{s}{c}\frac{F_c'(s)}{F_c(s)}.
\end{equation}
If the nonconstant polynomial $F_c$ had a zero $\zeta\ne0$ of multiplicity
$m$, then $q$ would be rational and
\begin{equation*}
 q(s)=\frac{m\zeta}{c(s-\zeta)}+O(1).
\end{equation*}
The germ relation $Q_k(q(s))=s^k$, where
$Q_k(q)=k\int_0^q u^{k-1}P(u/k)\,\dd u$, would then be an identity of
rational functions.  Its left side would have a pole at $\zeta$, because
$Q_k$ is a nonconstant polynomial, while its right side would be regular.
The constant case is impossible since $F_c'(0)=c$.  Thus the nonnegative
coefficient sequence has unbounded positive support.

If some $\lambda_j<0$, choose $n$ with $f_{c,n}>0$ so large that
$c+(n/k)\lambda_j<0$, contradicting
\eqref{eq:inverse-profile-positivity}.  Hence every $\lambda_j\geq0$.
Suppose that $\lambda_{\max}>0$, set
$q_*=-k/\lambda_{\max}$, and observe that
\begin{equation}\label{eq:Q-ratio}
 \frac{Q_k(q)}{q^k}
 =k\int_0^1v^{k-1}P(qv/k)\,\dd v>0
 \qquad(q_*\leq q<0).
\end{equation}
Define, using the positive real $k$th root,
\begin{equation}\label{eq:uk-root-map}
 S_k(q)=q\left(\frac{Q_k(q)}{q^k}\right)^{1/k},
 \qquad S_k(0)=0,\qquad S_k'(0)=1.
\end{equation}
Then
\begin{equation*}
 S_k'(q)=P(q/k)\left(\frac{q}{S_k(q)}\right)^{k-1}>0
 \qquad(q_*<q\leq0).
\end{equation*}
Thus $S_k$ maps $[q_*,0]$ monotonically onto $[s_*,0]$ for a finite
$s_*<0$.  Since $S_k'(q)\to0$ as $q\to q_*^+$, the inverse satisfies
$q(s)\to q_*$ and $q'(s)\to+\infty$ as $s\to s_*^+$.  The identity
$h'=q/s$ and $s_*\ne0$ show that $h$ remains bounded on $(s_*,0]$, while
\begin{equation*}
 h''(s)=\frac{sq'(s)-q(s)}{s^2}\longrightarrow-\infty
 \qquad(s\to s_*^+).
\end{equation*}
For $q>0$, one has $P(q/k)>0$ and $S_k'(q)>0$, while
$S_k(q)\to\infty$.  The inverse-function theorem therefore continues
$q$ and $h$ holomorphically to an open neighbourhood of $(s_*,\infty)$.
The continued $h$ is real on this interval and satisfies $h(0)=0$.
\Cref{lem:one-variable}, with $\Psi=h$ and $R=-s_*$, contradicts the
nonnegativity of the coefficients of $F_c$.
Therefore all
$\lambda_j=0$.  The profile equation gives $q=s$, and the normalization
then gives $h=s$ and $F_c=\ee^{cs}$.
\end{proof}

We can now combine the geometric blocks with the scalar rigidity
lemma.

\begin{proof}[Proof of \cref{thm:uk-rigidity}]
The compatibility condition and the profile follow from
\cref{lem:uk-determinant} and
\cref{prop:uk-existence}.  Since $W$ is connected and
$\widetilde\omega_k$ is real analytic, \cref{lem:global-character}
makes $c\widetilde\omega_k$ projectively induced near every point $(p,0)$
of the zero section.  \Cref{cor:uk-blocks} then supplies the
hypotheses of \cref{lem:positive-inverse-profile}, with $\lambda_j$
the eigenvalues of $A$.  Hence $A=0$, $q=s$, and $h=s$.  The induced
Hermitian metric on $L$ is locally flat, and \eqref{eq:uk-product} follows.

For the converse, let $\rho=0$.  Then $h(s)=s$ and $f_{c,n}=c^n/n!>0$ for
every $n$.  Moreover $\log\ell$ is pluriharmonic, so the normalization
$\ell(z,\bar p)\equiv1$ forces $\ell\equiv1$ and $K_n=\ee^{c\phi}$ for
every $n$.  As $C(\ee^{c\phi})$ is the direct sum of $[1]$ and
$C(\ee^{c\phi}-1)$, conditions (ii) and (iii) of
\cref{prop:fibre-degree} reduce to~(i): $c\widetilde\omega_k$
is projectively induced near $(p,0)$ exactly when $c\omega$ is projectively
induced near $p$.  Equivalently, an immersion is the Segre product of a
local immersion of $(X,c\omega)$ with the Bargmann--Fock immersion
$\boldsymbol\xi\mapsto(\sqrt{c^{|\alpha|}/\alpha!}\,\boldsymbol\xi^\alpha)_\alpha$
of flat $\mathbb C^k$, where $\alpha$ ranges over all multi-indices.  Since infinitely many $f_{c,n}$ are positive,
\eqref{eq:fibre-rank} gives infinite rank, so every realizing
projective target is infinite-dimensional.
\end{proof}

In particular, a counterexample of minimal complex dimension to the
Loi--Salis--Zuddas conjecture cannot arise as a nontrivial extension of
this form.

\section{Complete Radial Models}
\label{sec:canonical}

A complete metric containing the germ of \cref{thm:uk-rigidity} is
excluded as soon as the zero-section Ricci form is nonzero, by
\cref{lem:global-character}.  We first identify the resolved conifold
and then record the other complete constructions to which this applies.

\begin{example}[The Resolved Conifold]
\label{ex:resolved-conifold}
Take $k=2$, $X=\mathbb{CP}^1$, $L=\mathcal O(-1)$, and
\begin{equation*}
 \omega=4a^2\omega_{\mathrm{FS}},\qquad
 \rho(\omega_{\mathrm{FS}})=2\omega_{\mathrm{FS}},\qquad a>0.
\end{equation*}
Then $L^{\otimes2}=K_X$, the Ricci endomorphism is
$A=(2a^2)^{-1}I$, and
\begin{equation*}
 P(u)=1+\frac{u}{2a^2}.
\end{equation*}
\cref{prop:uk-existence} gives
\begin{equation}\label{eq:resolved-conifold-profile}
 s^2=q^2+\frac{q^3}{6a^2}.
\end{equation}
On the standard chart of
$\operatorname{Tot}(\mathcal O(-1)^{\oplus2})$, rescale the root
isomorphism---equivalently, apply the permitted fibre dilation---and use
the normalized radial variable
\begin{equation*}
 r^2=(1+|z|^2)(|\xi_1|^2+|\xi_2|^2),\qquad
 s=\frac{r^2}{\sqrt6\,a}.
\end{equation*}
Writing the Candelas--de la Ossa potential as
$4a^2\log(1+|z|^2)+F(r^2)$ and
$\gamma=r^2F'(r^2)$, set $h(s)=F(\sqrt6\,a\,s)$; then
$q=sh'(s)=\gamma$.
Equation~\eqref{eq:resolved-conifold-profile} is therefore exactly
\begin{equation*}
 \gamma^3+6a^2\gamma^2-r^4=0,
\end{equation*}
the defining cubic of the resolved-conifold metric
\cite[Section~4, (4.17)--(4.22)]{CandelasDeLaOssa1990}.  Thus the complete
Candelas--de la Ossa metric is the continuation of the normalized radial
germ in this case.
\end{example}

For $k=1$, Bielawski constructs a circle-invariant Ricci-flat germ on
$K_X$ over every real-analytic K\"ahler base, uniquely up to fibre
dilation \cite[Theorem~1]{Bielawski2002}.  When the base has constant
principal Ricci curvatures, \cref{prop:uk-existence} identifies this
germ with the profile above.  This observation supplies the rank-one
cases of the following corollary.

\Needspace{8\baselineskip}
\begin{corollary}[Classical Complete Models]
\label{cor:classical-complete}
For every positive multiple of each of the following complete Ricci-flat
K\"ahler metrics, no nonempty open subset admits a K\"ahler immersion into
$\mathbb{CP}^N$, $N\leq\infty$:
\begin{enumerate}[label=\textup{(\roman*)}]
\item the Candelas--de la Ossa metric on
 $\operatorname{Tot}(\mathcal O_{\mathbb{CP}^1}(-1)^{\oplus2})$;
\item the Pando Zayas--Tseytlin metrics on
 $K_{\mathbb{CP}^1\times\mathbb{CP}^1}$, including those in which the two
 factors have different radii;
\item van Coevering's asymptotically conical metrics in arbitrary
 invariant K\"ahler classes on $K_Z$, where $Z$ is a
 positive-dimensional flag variety;
\item the Achmed-Zade--Bykov metrics in arbitrary invariant K\"ahler
 classes on $\operatorname{Tot}((K_F^{1/k})^{\oplus k})$, over the
 positive-dimensional $SU(n)$ flag manifolds $F$ treated in
 \cite{AchmedZadeBykov2020}, whenever the indicated root exists and
 $k\geq2$.
\end{enumerate}
\end{corollary}

\begin{proof}
Part~\textup{(i)} follows from \cref{ex:resolved-conifold}.  The
Pando Zayas--Tseytlin metrics are circle-invariant metrics on
$K_{\mathbb{CP}^1\times\mathbb{CP}^1}$
\cite{PandoZayasTseytlin2001}; their canonical-bundle interpretation and
independently scaled Fubini--Study zero-section forms are recorded
explicitly in~\cite[Section~2]{AchmedZadeBykov2020}.  The invariant
zero-section metrics in van Coevering's construction are homogeneous
\cite{vanCoevering2018Flag}.  In both cases the Ricci endomorphism has
constant, nonzero spectrum, and Bielawski's uniqueness theorem identifies
each germ, after fibre dilation, with the $k=1$ profile constructed in
\cref{prop:uk-existence}.
Achmed-Zade and Bykov write $q$ for the number of summands, our $k$, and
$x_0,\ldots,x_{q-1}$ for the fibre coordinates.  In their notation,
equations~(3.24) and (3.35)--(3.37) of \cite{AchmedZadeBykov2020} give
\begin{equation*}
 \ee^{t/q}
 =\left(\sum_{m=0}^{q-1}|x_m|^2\right)
   \prod_j t_j^{c_j/q},
\end{equation*}
where $t$ is the variable through which their potential depends on the
fibre coordinates, the $t_j$ are functions on the base, and the $c_j$ are
constants.  Thus the potential depends on the fibre coordinates only
through their squared norm, and the construction is $U(k)$-radial over a
homogeneous base.  Its germ along the zero section is the profile of
\cref{prop:uk-existence} with $k=q$ and, for a
positive-dimensional flag variety, again has nonzero Ricci form.
In every case the total space is connected, so by
\cref{lem:global-character} a K\"ahler immersion of a nonempty open
subset would give one near the zero section; since the zero-section germ
has nonzero Ricci form, \cref{thm:uk-rigidity} excludes this.
Completeness is supplied by the cited constructions, not by the local
theorem.
\end{proof}

When their zero-section metrics are K\"ahler--Einstein, the rank-one
metrics in parts (ii) and (iii) are Calabi's canonical-bundle metrics,
whose noninducibility was proved in \cite{LoiZeddaZuddas2021}.
The additional cases here include unequal-radius products and arbitrary
invariant classes over homogeneous bases.

\section{Applications Along Holomorphic Curves}
\label{sec:curves}

We now apply \cref{prop:radial-curve} to explicit holomorphic curves in
complete Ricci-flat manifolds.  Each application identifies the induced
radial potential and checks the hypotheses of \cref{lem:one-variable};
no radial bundle presentation of the ambient metric is required.

\subsection{The Taub--NUT Metric}

The axis of LeBrun's realization of Taub--NUT on $\mathbb C^2$ provides
a direct application.  For $m\geq0$, define $u,v\geq0$ by
\begin{equation}\label{eq:lebrun-coordinates}
 |z_1|=\ee^{m(u^2-v^2)}u,\qquad |z_2|=\ee^{m(v^2-u^2)}v.
\end{equation}
The metric $g_m$ has K\"ahler form
$\frac i2\partial\bar\partial\Phi_m$, where
$\Phi_m=u^2+v^2+m(u^4+v^4)$; for $m>0$ it is complete, Ricci-flat
and nonflat \cite{LeBrun1991Cn,LoiZeddaZuddas2012}.
In our convention its potential is $\Phi_m/2$.  Since every positive
multiple is considered, this constant factor does not affect the
obstruction.  On $z_2=0$, put $U=u^2$ to obtain
\begin{equation}\label{eq:taub-nut-axis}
 |z_1|^2=U\ee^{2mU},\qquad \Phi_m|_{z_2=0}=U+mU^2.
\end{equation}
The inverse profile has the finite negative endpoint needed above.

\begin{theorem}[The Taub--NUT Metric]
\label{thm:taub-nut}
Let $m>0$.  For every $c>0$, no nonempty open subset of
$(\mathbb C^2,c\,g_m)$ admits a K\"ahler immersion into $\mathbb{CP}^N$,
$N\leq\infty$; already the metric induced by $c\,g_m$ on the axis
$\{z_2=0\}$ fails Calabi's criterion at the origin.  In particular no
positive multiple of the Taub--NUT metric is projectively induced.
\end{theorem}

\begin{proof}
With $W=v^2$, the map $(U,W)\mapsto(U\ee^{2m(U-W)},W\ee^{2m(W-U)})
=(|z_1|^2,|z_2|^2)$ has Jacobian determinant $1+2m(U+W)>0$ on a
neighbourhood of $[0,\infty)^2$, so $\Phi_m$ is a real analytic function
of $(|z_1|^2,|z_2|^2)$ there and $g_m$ is a real analytic K\"ahler metric
on $\mathbb C^2$.  The function $h(U)=U\ee^{2mU}$ has
$h'(U)=(1+2mU)\ee^{2mU}>0$ for $U>-1/(2m)$, so it restricts to an
increasing real analytic bijection of $(-1/(2m),\infty)$ onto
$(-R,\infty)$, where $R=1/(2m\ee)$.  Let $U(X)$ be its inverse and put
$\Psi(X)=U(X)+mU(X)^2$.  By \eqref{eq:taub-nut-axis}, the axis
$\{z_2=0\}$ is an embedded disc through the origin on which the K\"ahler
form of $g_m$ is $\frac i2\partial\bar\partial\Psi(|z_1|^2)$; since the
hypotheses of \cref{lem:one-variable} are invariant under positive
multiples of $\Psi$, \cref{prop:radial-curve} applies once they are
verified for $\Psi$.  Now $\Psi$ is real analytic on $(-R,\infty)$, hence
holomorphic on a connected open neighbourhood of this half-line in
$\mathbb C$, with $\Psi(0)=0$, and it is bounded on $(-R,0]$, where $U$
increases from $-1/(2m)$ to $0$.  Differentiating $h(U)=X$ gives
$U'=\ee^{-2mU}/(1+2mU)$, hence
\begin{equation*}
 \Psi'=(1+2mU)U'=\ee^{-2mU},\qquad
 \Psi''=-2mU'\ee^{-2mU}=-\frac{2m\,\ee^{-4mU}}{1+2mU}\longrightarrow-\infty
 \qquad(X\downarrow-R),
\end{equation*}
because $U\to-1/(2m)$.  Thus $\Psi$ satisfies the hypotheses of
\cref{lem:one-variable}, with $\Psi''$ unbounded on $(-R,0]$, and
\cref{prop:radial-curve}, applied to the axis with $\lambda=1$, gives the
conclusion for $\mathbb{CP}^\infty$, hence for every $N\leq\infty$; its
proof shows that already the metric induced on the axis fails Calabi's
criterion at the origin.
\end{proof}

\subsection{Calabi's Hyperk\"ahler Metrics}

A line in a cotangent fibre gives the same obstruction for Calabi's
hyperk\"ahler metric, although the cotangent bundle is not a direct sum
of line bundles when $n\geq2$.  Write
$\pi:T^{*1,0}\mathbb{CP}^n\to\mathbb{CP}^n$ for the bundle projection.

\Needspace{7\baselineskip}
\begin{proposition}[Calabi's Hyperk\"ahler Metrics]
\label{prop:calabi-hk}
Let $\kappa>0$, let $\mathbb{CP}^n$, $n\geq1$, carry the Fubini--Study
metric of holomorphic sectional curvature $2\kappa$ with local potential
$\Phi$, and let
$t$ be the induced squared norm on the holomorphic cotangent bundle
$T^{*1,0}\mathbb{CP}^n$.  Calabi's complete hyperk\"ahler metric on
$T^{*1,0}\mathbb{CP}^n$ has K\"ahler potential $\Phi\circ\pi+u\circ t$
with
\begin{equation*}
 u(t)=\kappa^{-1}\bigl(\sqrt{1+4\kappa t}
      -\log(1+\sqrt{1+4\kappa t})\bigr)
\end{equation*}
\cite[Th\'eor\`eme~5.3]{Calabi1979}.  Here $\kappa$ corresponds to Calabi's
curvature constant $K$.  For every $c>0$, no nonempty open subset of
$T^{*1,0}\mathbb{CP}^n$, with this complex structure and $c$ times this
metric, admits a K\"ahler immersion into $\mathbb{CP}^{\infty}$.
\end{proposition}

\begin{proof}
Fix $p\in\mathbb{CP}^n$ and linear coordinates $\zeta$ on the fibre
$T^{*1,0}_p\mathbb{CP}^n$ in which $t=|\zeta|^2$.  The line
$\sigma\mapsto\sigma\zeta_0$, $|\zeta_0|=1$, in this fibre carries the
potential $u(|\sigma|^2)+\Phi(p)$, and
$u(|\sigma|^2)-u(0)=\Psi(4\kappa|\sigma|^2)$ with
\begin{equation*}
 \Psi(x)=\kappa^{-1}\bigl(\sqrt{1+x}-\log(1+\sqrt{1+x})-1+\log2\bigr).
\end{equation*}
With principal branches, $\Psi$ is holomorphic on
$\mathbb C\setminus(-\infty,-1]$, real on $(-1,\infty)$, bounded on
$(-1,0]$, and $\Psi(0)=0$.  Moreover,
\begin{equation*}
 \Psi'(x)=\frac{1}{2\kappa(1+\sqrt{1+x})},\qquad
 \Psi''(x)=-\frac{1}{4\kappa\sqrt{1+x}(1+\sqrt{1+x})^2}.
\end{equation*}
The second derivative is unbounded as $x\to-1^+$.
\Cref{prop:radial-curve} applies with
$\lambda=4\kappa$ and $R=1$.
\end{proof}

For $n=1$ this is the Eguchi--Hanson metric, already covered by
\cite{LoiZeddaZuddas2021}.  The potential $u$ is independent of $n$, so
every fibre line in the proof carries the same induced metric as in
that case.

\subsection{Cifarelli's Type-2 Metric}

Cifarelli's Type-2 construction gives another complete Ricci-flat metric
on the resolved conifold, with volume growth of order $R^5$ in real
dimension six \cite[Theorem~B and Example~1.5]{Cifarelli2024}.
In his notation, $\alpha_j$ are endpoints of the separation variables,
$F_j$ are their metric coefficients, and $q_{\mathrm C}$ is the
polynomial in the Ricci-flat equation.

\begin{corollary}[Cifarelli's ALF-like Type-2 Metric]
\label{cor:cifarelli-type2}
Let $M=\operatorname{Tot}(\mathcal O_{\mathbb{CP}^1}(-1)^{\oplus2})$, and
let $g_{\mathrm{Cif}}$ be the complete Ricci-flat K\"ahler metric obtained
from the Type-2 Hamiltonian $2$-form data
\begin{equation}\label{eq:cifarelli-type2-data}
 \alpha_1=-1,\qquad \alpha_2=1,\qquad q_{\mathrm C}(t)=2,
 \qquad F_1(t)=F_2(t)=2(t-t^{-1}).
\end{equation}
These are the data, normalized so that the metric extends smoothly, of
the ALF-like metric on the resolved conifold in
\cite[Theorem~B(a) and Example~1.5]{Cifarelli2024}.  For
every $c>0$, no nonempty open subset of $(M,cg_{\mathrm{Cif}})$ admits a
K\"ahler immersion into $\mathbb{CP}^N$, $N\leq\infty$.
\end{corollary}

\begin{proof}
We first check the normalization and then compute on a coordinate fibre
line.  For this base and two line summands, the Ricci-flat and smooth-end
conditions of \cite[Sections~2 and~3.3]{Cifarelli2024} reduce to
\begin{equation*}
 F_j'(t)+t^{-1}F_j(t)=4=2q_{\mathrm C}(t),\qquad
 F_1(-1)=F_2(1)=0,\qquad
 F_1'(-1)=F_2'(1)=4.
\end{equation*}
They also have the Type-2 signs $F_1<0$ on $(-\infty,-1)$ and $F_2>0$
on $(1,\infty)$.  The two compactification degrees are
\begin{equation*}
 \delta_1=\frac{2}{-2}=-1,
 \qquad
 \delta_2=\frac{-2}{2}=-1,
\end{equation*}
so the compactification is $\mathcal O(-1)^{\oplus2}$.  The displayed
functions satisfy these identities and have linear growth at infinity.
The endpoint, positivity, and growth criteria in
\cite[Lemmas~2.2, 2.7, and~2.10 and Proposition~2.13]{Cifarelli2024},
together with its Ricci-flatness criterion in Lemma~3.3, give the asserted
smooth complete metric.

Let $\xi_1\in(-\infty,-1)$ and $\xi_2\in(1,\infty)$ be the two separation
variables, and put
\begin{equation*}
 p_{\mathrm{nc}}(t)=(t-\xi_1)(t-\xi_2).
\end{equation*}
The nonnegative moments of the
two fibre coordinates are
$\mu_j=-\tfrac12p_{\mathrm{nc}}(\alpha_j)$.  For the bundle
projection $\pi:M\to\mathbb{CP}^1$, fix $b\in\mathbb{CP}^1$ and let
$C_b\subset M_b:=\pi^{-1}(b)\cong\mathbb C^2$ be the
first coordinate line, so that $C_b=\{\mu_2=0\}\cap M_b$.  This is the
facet $\xi_2=1$.  Write
\begin{equation*}
 \xi_1=-1-u,\qquad u\geq0.
\end{equation*}
Use local angular coordinates $\tau_1,\tau_2$ conjugate to the moment
coordinates $\sigma_1=\xi_1+\xi_2$ and $\sigma_2=\xi_1\xi_2$.
The primitive generator of the noncollapsed circle is
$v_1=-\tfrac12(\partial_{\tau_1}+\partial_{\tau_2})$, so its angle
$\theta$ has period $2\pi$ \cite[equation~(2.16)]{Cifarelli2024}.
Substitution in the Hamiltonian $2$-form metric
\cite[equation~(2.2)]{Cifarelli2024} gives
\begin{equation}\label{eq:cifarelli-axis-metric}
 g_{\mathrm{Cif}}|_{C_b}
 =\frac{u+1}{2u}\,(\dd u)^2+\frac{2u}{u+1}\,(\dd\theta)^2,
 \qquad
 \omega_{\mathrm{Cif}}|_{C_b}=\dd u\wedge\dd\theta.
\end{equation}
Indeed,
$F_1(-1-u)=-2u(u+2)/(u+1)$ and
$\Delta(\xi_1)=\xi_1-\xi_2=-(u+2)$.  The formula for the K\"ahler form
follows from $\sigma_1=-u$ and $\sigma_2=-1-u$.
In particular, with $R=\sqrt{2u}$,
\eqref{eq:cifarelli-axis-metric} becomes
\begin{equation*}
 \left(1+\frac{R^2}{2}\right)(\dd R)^2
 +\frac{R^2}{1+R^2/2}(\dd\theta)^2,
\end{equation*}
which is smooth at the origin because $\theta$ has period $2\pi$.

Let $z$ be the weight-one holomorphic coordinate on this line.  If
$x=\lambda|z|^2$, with $\lambda>0$ absorbing a fibre dilation, K\"ahler
compatibility in \eqref{eq:cifarelli-axis-metric} gives
\begin{equation*}
 \frac{\dd\log x}{\dd u}=\frac{u+1}{u},
 \qquad\text{and hence}\qquad x=u\ee^u.
\end{equation*}
Thus $u=W_0(x)$, where $W_0$ is the principal branch of the Lambert $W$
function \cite{CorlessEtAl1996}.  In our
$i\partial\bar\partial$ convention a radial potential $\Psi(x)$ satisfies
$\omega=\dd(x\Psi'(x))\wedge\dd\theta$.  Since the moment is $u$, the
normalized potential on the line is therefore
\begin{equation}\label{eq:cifarelli-axis-potential}
 \Psi(x)=W_0(x)+\frac12W_0(x)^2.
\end{equation}
On $\mathbb C\setminus(-\infty,-1/\ee]$ this function is holomorphic, and
it is real on $(-1/\ee,\infty)$.  Moreover,
\begin{equation*}
 \Psi(0)=0,
 \qquad \lim_{x\to-1/\ee^+}\Psi(x)=-\frac12,
 \qquad
 \Psi''(x)=-\frac{\ee^{-2W_0(x)}}{1+W_0(x)}\longrightarrow-\infty.
\end{equation*}
The total space $M$ is connected, and $g_{\mathrm{Cif}}$, being
Ricci-flat K\"ahler, is real analytic in holomorphic coordinates, whose
real and imaginary parts are harmonic \cite{DeTurckKazdan1981}.  Thus all ambient hypotheses of
\cref{prop:radial-curve} hold.
\Cref{prop:radial-curve} applies with $R=1/\ee$, using the
unbounded second derivative.
\end{proof}

The endpoint computation fixes the normalization used here.\footnote{The
normalization here follows equation~(3.3) and Lemma~2.2 of
\cite{Cifarelli2024}.  In that convention, the integral in arXiv v2,
(3.17), and hence the displayed functions in (3.19)--(3.20), require an
additional factor $2$.}
The function \eqref{eq:cifarelli-axis-potential} equals
$\Phi_{1/2}$ restricted to the Taub--NUT axis in
\eqref{eq:taub-nut-axis}.  With the conventions used here, the induced
metric is twice the Taub--NUT axis metric, after fibre dilation.
Thus the new geometric
input in this application is the identification of the coordinate fibre
line and its induced metric.

\section{Further Directions}
\label{sec:directions}

The main open extension is to potentials
$\phi(x)+H(x,s)$ whose fibre profile depends on the base point $x$.
Bielawski's general canonical-bundle germs have this form.
For a fixed $c>0$, write
\begin{equation*}
 \exp(cH(x,s))=\sum_{n\geq0}b_{c,n}(x)s^n.
\end{equation*}
On a region where $b_{c,n}>0$, the curvature of the degree-$n$ block contains the additional
term $i\partial\bar\partial\log b_{c,n}$.  Controlling it would extend the
Ricci-sign argument beyond constant principal Ricci curvatures.
The remaining Ricci-flat base problem is also essential: the product
case of \cref{thm:uk-rigidity} reduces dimension but does not settle the
Loi--Salis--Zuddas conjecture.

\section*{Acknowledgments and Tool Disclosure}

\noindent\textit{AI-use disclosure.}
The author used Anthropic Claude Code and OpenAI Codex as interactive
research and writing tools.  These tools assisted with exploratory
discussion, testing and refinement of ideas, literature and source
organization, code development and verification, mathematical error
checking, and editorial revision.  The author takes full responsibility
for the content presented.

\printbibliography[heading=bibintoc]

\end{document}